\title{Geometric interpretation of Hermitian modular forms via Burkhardt invariants}
\author{Atsuhira Nagano and Hironori Shiga}

\documentclass[10pt]{article}
\usepackage{mathrsfs,bm,graphics,amsfonts,amsthm,amssymb,amsmath,amscd,booktabs}
\usepackage[dvips]{graphicx,color}
\def\bigzerou{\smash{\lower1.7ex\hbox{\b 0}}}

\newtheorem{thm}{Theorem}[section]
\newtheorem{df}{Definition}[section]
\newtheorem{lem}{Lemma}[section]
\newtheorem{prop}{Proposition}[section]
\newtheorem{rem}{Remark}[section]

\newtheorem{cor}{Corollary}[section]

\def\comment#1{{ }}
\makeatletter
  
      \@addtoreset{equation}{section}
 \usepackage{fancybox}
\begin{document}
\maketitle
\setlength{\baselineskip}{13 pt}
 \renewcommand{\thefootnote}{\fnsymbol{footnote}}

\begin{abstract}
We  give a complete theta expression of a pair of Hermitian modular forms  as an inverse period mapping of lattice polarized $K3$ surfaces.
Our result gives a non-trivial relation among moduli of $K3$ surfaces, theta functions and the finite complex reflection group of rank $5$.
\end{abstract}

\footnote[0]{Running head: Geometric interpretation of Hermitian modular forms}
\setlength{\baselineskip}{14 pt}

\section*{Introduction}
Hermitian modular forms are automorphic forms on bounded symmetric domains of type $I$.
We can regard them as  natural extensions of elliptic, Hilbert or Siegel modular forms.
Classically, elliptic modular forms are derived from the moduli of elliptic curves.
Namely, an elliptic curve $\mathbb{C}/(\mathbb{Z}+\mathbb{Z} \tau)$, where $\tau\in \mathbb{H}$ is  called the period, can be given  by the Weierstrass equation
\begin{align}\label{EllipticWeierstrass}
y^2 = 4x^3 -g_2(\tau) x -g_3(\tau)
\end{align}
and $ g_2(\tau)$ and $g_3(\tau)$ give elliptic modular functions.
Also, Hilbert and Siegel modular forms of degree $2$ are derived from the moduli of Abelian surfaces.
However, we do not have canonical  moduli interpretations of Hermitian modular forms.
In this paper, 
by using lattice polarized $K3$ surfaces,
we will give a natural moduli interpretation of Hermitian modular forms, which are closely related  to complex reflection groups.

Let us recall the works \cite{SI}, \cite{NHilb} and \cite{CD}.
In those works, 
they give elliptic, Hilbert and Siegel modular forms respectively,
by using the inverse correspondences of period mappings of $K3$ surfaces.
There are several good applications of their results in number theory and arithmetic geometry.
We will obtain a natural extension of those works.
By virtue of the Torelli theorem and the surjectivity of the period mappings,
we can consider moduli of $K3$ surfaces via period domains.
The period domain corresponding to  a family of lattice polarized $K3$ surfaces of Picard number $\rho$
 is given by the bounded symmetric domain of type $IV$ attached to the orthogonal group $SO_0 (2,20-\rho)$.
This implies that
we can obtain Hermitian modular forms by using $K3$ surfaces of Picard number $16$,
since there is an isomorphism
$SO_0(2,4) \simeq SU(2,2)$.
In this paper,
we will use a family of $K3$ surfaces of (\ref{S(t)}) studied in \cite{N}.
This family  naturally contains the families of \cite{SI}, \cite{NHilb} and \cite{CD}
and
is characterized by the lattice of (\ref{LatticeA}),
which is the simplest lattice satisfying the Kneser condition.
The purpose of the present paper 
is to determine
 the complete expression of the inverse period mapping of the family (\ref{S(t)})
in Hermitian modular forms
by using  theta functions.

Here, theta functions give us a combinatorial viewpoint.
Finite complex reflection groups are very important in combinatorics
and  classified into 37 types by Shephard-Todd  \cite{ST}.
In the works \cite{SI}, \cite{NHilb},  \cite{CD} and this paper, 
 the modular forms are constructed 
 by using the appropriate invariants for complex reflection groups  
and the theta functions in each case (see Table 1).
For example, 
the  inverse period mapping in \cite{CD} 
gives a pair of Siegel modular forms of degree $2$ (see Proposition \ref{PropCDIgusa}).
We notice that,
by virtue of the result of Runge \cite{R} Section 4,
 this pair 
has a theta expression using the invariants for the action of the reflection group $G_{31}$  in \cite{ST}  on $\mathbb{P}^3(\mathbb{C})$.
In our argument,
we  will apply the action of the  reflection group $G_{33}$,
which is the unique exceptional group of rank $5$ in \cite{ST},
on $\mathbb{P}^4(\mathbb{C})$.
The quotient group $G_{33}/\{\pm 1\}$ is called the Burkhardt group, which is precisely studied in \cite{B} and \cite{FSM}. 
Our main result is expressed
 in terms of  the Burkhardt invariants  and the theta functions which are investigated by  \cite{DK} and  \cite{FSMTheta}.

\begin{table}[h]
\center
\begin{tabular}{ccccc}
\toprule
   Modular Forms    & Shephard-Todd Groups  &Theta Functions &  $K3$ Surfaces  \\
 \midrule
Elliptic &  No.8 &  Classical  & \cite{SI} \\
Hilbert & No.23  &  \cite{Muller}  & \cite{NHilb} \\
Siegel & No.31  &  \cite{I35}  & \cite{CD}\\
Hermitian & No.33 &   \cite{DK}  & This Paper \\
 \bottomrule
\end{tabular}
\caption{Modular forms, complex reflection groups, theta functions and  $K3$ surfaces}
\end{table}

The contents are as follows.
In section 1, we will survey the results of \cite{N}.
The period mapping is described  by the  period integrals of $K3$ surfaces.
In section 2, by using the geometry of elliptic surfaces, 
we will give an explicit construction of our period integrals (Theorem \ref{ThmGeomCycle}). 
In section 3, we will see basic properties of Hermitian modular forms.
In section 4,
after we  survey  properties of theta functions and the Burkhardt group,
we will  determine the exact expression of Hermitian modular forms as the inverse period mapping 
 by considering the  degenerations of our $K3$ surfaces (Theorem \ref{ThmThetaExpression}).

Thus,
we will give a geometric meaning of Hermitian modular forms
by using a sequence of the families  of $K3$ surfaces, 
 which are concordant with theta functions and  complex reflection groups, as shown in Table 1.
 Such constructions of modular forms must be useful in number theory from   theoretical and computational viewpoints.
Theta functions and $\mathbb{Q}$-rational modular forms have  many fruitful applications in number theory (for example, see \cite{Shimura} Chapter VII).
Our result gives a non-trivial relation
between the periods of  $K3$ surfaces and  
 the pair of $\mathbb{Q}$-rational Hermitian modular forms (Corollary \ref{CorQ-rational}).
Hence,  
our family of $K3$ surfaces will give an explicit  geometric model
with good  arithmetic properties, 
just as the family of elliptic curves (\ref{EllipticWeierstrass}) does.
Furthermore, 
our $K3$ surfaces do not admit the Shioda-Inose structure.
This fact means that 
the moduli space of our family of $K3$ surfaces beyonds 
that of the family of principally polarized Abelian surfaces.
Therefore, 
the authors expect that our results will provide non-trivial applications in number theory and arithmetic geometry.
For example,
it is expected that our results will give a good test case for the theory of complex multiplication of $K3$ surfaces (for example, see \cite{S}).

\section{Lattice polarized $K3$ surfaces}

%\subsection{The family $\{S([t])\}$ of $K3$ surfaces}

   In this section, we survey the results of \cite{N} about periods of a family of  lattice polarized  $K3$ surfaces.
For detailed proofs, see \cite{N}.
 
 Let $t=(t_4,t_6,t_{10},t_{12},t_{18})\in \mathbb{C}^5-\{0\}.$
 In \cite{N},
 we study the hypersurfaces
 \begin{align}\label{S(t)}
 S(t): z^2=y^3 + (t_4 x^4 w^4 + t_{10} x^3 w^{10})y + (x^7 + t_6 x^6 w^6 + t_{12} x^5 w^{12} + t_{18} x^4 w^{18})
 \end{align}
 of weight $42$
 in the weighted projective space 
 $\mathbb{P}(6,14,21,1)={\rm Proj} (\mathbb{C} [x,y,z,w])$.
 There is an action of the multiplicative group $\mathbb{C}^*$ 
 on $\mathbb{P}(6,14,21,1)$
 ($\mathbb{C}^5-\{0\}$, resp.)
 given by
 $(x,y,z,w) \mapsto (x,y,z,\lambda^{-1} w)$
 ($t=(t_4,t_6,t_{10},t_{12},t_{18}) \mapsto \lambda \cdot t =(\lambda^4 t_4, \lambda^6 t_6, \lambda^{10 } t_{10}, \lambda^{12 } t_{12}, \lambda^{18} t_{18})$, resp.)
  for $\lambda\in \mathbb{C}^*$.
So,
 we naturally obtain the family
 $$
\{S ([t]) \hspace{0.5mm} | \hspace{0.5mm} [t] \in \mathbb{P}(4,6,10,12,18) \} \rightarrow \mathbb{P}(4,6,10,12,18).
 $$
 Here, 
 the point of $\mathbb{P}(4,6,10,12,18)$ corresponding to $t=(t_4,t_6,t_{10},t_{12},t_{18}) \in \mathbb{C}^5 -\{0\}$
 is denoted by
 $[t]=(t_4:t_6:t_{10}:t_{12}:t_{18})$.
 
 Set
 \begin{align}
 T= \mathbb{P}(4,6,10,12,18) - \{[t] \in \mathbb{P}(4,6,10,12,18) \hspace{0.5mm} | \hspace{0.5mm} t_{10}=t_{12}=t_{18}=0 \}.
\end{align}
 Then, a member of the family
 $$
 \{S([t]) \hspace{0.5mm} | \hspace{0.5mm} [t]\in T \} \rightarrow T
 $$
 is a  $K3$ surfaces.
Here, let us consider the $\mathbb{C}^*$-bundle $T^*\rightarrow T$ naturally coming from the above action $t\mapsto \lambda\cdot t$ of $\mathbb{C}^*$.
We  also have the family 
 $$
 \{S(t) \hspace{0.5mm} | \hspace{0.5mm} t \in T^*\} \rightarrow T^*
 $$
 of $K3$ surfaces.
 The  action of $\mathbb{C}^*$ on $(x,y,z,w)$ and $t $ gives an isomorphism
 $
 \lambda : S(t) \mapsto S(\lambda\cdot t)
 $
 for any $\lambda\in \mathbb{C}^*.$
 
 There exists the unique holomorphic $2$-form  on a $K3$ surface up to a constant factor.
 By putting $x_0=\frac{x}{w^6},y_0=\frac{y}{w^{14}}, z_0=\frac{z}{w^{21}}$,
we have a holomorphic family $\{\omega_t\}_{t\in T^*}$,
where $\omega_t=\frac{dx_0 \wedge dy_0}{z_0}$,
 of holomorphic $2$-forms on the $K3$ surfaces $S(t)$.
Since the above action of $\mathbb{C}^*$ 
 transforms $\omega_t$ into $\lambda^{-1} \omega_t$,
 the above isomorphism $\lambda$ 
 gives the correspondence
 \begin{align}\label{omegalambda}
 \lambda^{*} \omega_{\lambda \cdot t} = \lambda^{-1} \omega_t.
 \end{align}
 Hence,
 we have
 the family $\{\omega_{[t]}\}_{[t]\in T}$
 of holomorphic $2$-forms
 for the family $\{S([t]) \hspace{0.5mm} | \hspace{0.5mm} [t]\in T \}$ of $K3$ surfaces.

 In this paper, let $U$ be the unimodular hyperbolic lattice of rank $2$.
 Also, let $A_j(-1)$ and $E_j(-1)$ be the root lattices.
 Then, the $K3$ lattice $L_{K3}$, which is isomorphic to the $2$-homology group of a $K3$ surface, is isometric to
 $II_{3,19}=U\oplus U\oplus E_8(-1) \oplus E_8(-1) \oplus E_8(-1).$
 By the way, the lattice
 \begin{align}\label{LatticeA}
 A=U \oplus U \oplus A_2(-1)
 \end{align}
 of signature $(2,4)$ is the simplest lattice satisfying the Kneser condition,
 which is an arithmetic condition of quadratic forms.
Our $K3$ surfaces are closely related to this lattice as follows.
 From the lattice $A$, we define the $4$-dimensional space
 $
\mathcal{D}_M= \{ [\xi] \in \mathbb{P} (A \otimes \mathbb{C}) \hspace{0.5mm} | \hspace{0.5mm}      (\xi, \xi) =0,  (\xi, \overline{\xi}) >0\}.
$
Let $\mathcal{D}$ be a connected component of $\mathcal{D}_M$.
Then, we have a subgroup
 $
 O^+(A) = \{\gamma \in O(A) \hspace{0.5mm} | \hspace{0.5mm} \gamma(\mathcal{D}) = \mathcal{D}\},
 $
 satisfying $O(A)/O^+(A) \simeq \mathbb{Z}/2\mathbb{Z}$.
Also, set
 $
  \tilde{O}(A) = {\rm Ker}(O(A) \rightarrow {\rm Aut}(A^\vee/A)),
 $
 where $A^\vee ={\rm Hom}(A,\mathbb{Z})$.
 Set 
 \begin{align}
 \Gamma = \tilde{O}^+(A) = O^+(A) \cap \tilde{O}(A).
 \end{align}

\begin{lem}\label{LemStableG}
$$
O^+(A) / \Gamma \simeq \mathbb{Z}/2\mathbb{Z}.
$$
\end{lem}

 \begin{proof}
 Let $\{a_1,a_2,b_1,b_2,v_1,v_2\}$ be a system of basis of $A$ of (\ref{LatticeA}).
 Here,
 each of $\{a_1,a_2\}$ and $\{b_1,b_2\}$ gives a system of basis of each direct summand $U$ .
  Also, $\{v_1,v_2\}$ is a system of basis of $A_2(-1)$ satisfying $(v_1\cdot v_1) = (v_2\cdot v_2)=-2$ and $(v_1\cdot v_2)=1.$
 Then, setting
 $
 y_1=\frac{1}{3}v_1 +\frac{2}{3}v_2 \mod A 
 $
 and
 $
 y_2=\frac{2}{3} v_1 +\frac{1}{3}v_2 \mod A,
 $
 the set $A^\vee /A$ is just  $\{0,y_1,y_2\}$.
 This yields $O^+(A)/\Gamma \simeq \mathbb{Z}/2\mathbb{Z}.$
 \end{proof}

 Set 
 $$
 \mathcal{T}=T-(\{t_{18}=0\}\cup \{d_{90}(t)=0\}).
 $$
Here, 
  $d_{90}(t)$ is the following polynomial in $t_4,t_6,t_{10},t_{12} $ and $t_{18}$:
\begin{align}\label{d90}
 d_{90} (t) = & 3125 t_{10}^9 + 11664 t_{10}^3 t_{12}^5 + 151875 t_{10}^6 t_{12} t_{18} + 
  314928 t_{12}^6 t_{18} + 1968300 t_{10}^3 t_{12}^2 t_{18}^2 + 
  4251528 t_{12}^3 t_{18}^3\notag\\
  &+ 14348907 t_{18}^5 + 16200 t_4 t_{10}^5 t_{12}^3  + 
  472392 t_4 t_{10}^2 t_{12}^4 t_{18}  - 273375 t_4 t_{10}^5 t_{18}^2  - 
  5314410 t_4 t_{10}^2 t_{12} t_{18}^3  \notag \\
  &+ 4125  t_4^2 t_{10}^7 t_{12}  + 
  108135 t_4^2 t_{10}^4 t_{12}^2 t_{18}  - 1259712 t_4^2 t_{10} t_{12}^3 t_{18}^2  + 
  4251528 t_4^2 t_{10} t_{18}^4  + 864 t_4^3 t_{10}^3 t_{12}^4  \notag \\ 
  &
   - 
  3525 t_4^3 t_{10}^6 t_{18}  + 23328 t_4^3 t_{12}^5 t_{18} 
  - 
  378108t_4^3 t_{10}^3 t_{12} t_{18}^2  + 1102248 t_4^3 t_{12}^2 t_{18}^3
  + 
  888 t_4^4 t_{10}^5 t_{12}^2   \notag \\
  & + 26568 t_4^4 t_{10}^2 t_{12}^3 t_{18}  + 
  227448 t_4^4 t_{10}^2 t_{18}^3  + 16  t_4^5 t_{10}^7 - 456 t_4^5 t_{10}^4 t_{12} t_{18}  - 
  85536 t_4^5 t_{10} t_{12}^2 t_{18}^2\notag \\
  &  + 16 t_4^6 t_{10}^3 t_{12}^3  + 
  432 t_4^6 t_{12}^4 t_{18}  - 1056 t_4^6 t_{10}^3 t_{18}^2  + 62208 t_4^6t_{12} t_{18}^3  + 
  16 t_4^7 t_{10}^5 t_{12} + 480 t_4^7 t_{10}^2 t_{12}^2 t_{18}  \notag \\
  &  - 16 t_4^8 t_{10}^4 t_{18} - 
  1536 t_4^8 t_{10} t_{12} t_{18}^2  + 1024 t_4^9 t_{18}^3  - 13500 t_6 t_{10}^6 t_{12}^2  - 
  481140 t_6 t_{10}^3 t_{12}^3 t_{18} - 2834352   t_6 t_{12}^4 t_{18}^2 \notag \\ 
  &  - 
  1476225 t_6 t_{10}^3 t_{18}^3 - 19131876 t_6 t_{12} t_{18}^4  - 5625 t_4 t_6 t_{10}^8 - 
  200475  t_4 t_6 t_{10}^5 t_{12} t_{18} - 236196 t_4 t_6 t_{10}^2 t_{12}^2 t_{18}^2 \notag \\
  &  - 
  2592  t_4^2 t_6 t_{10}^4 t_{12}^3 - 69984 t_4^2 t_6 t_{10} t_{12}^4 t_{18}  + 
  422820 t_4^2 t_6  t_{10}^4 t_{18}^2 + 944784  t_4^2 t_6 t_{10} t_{12} t_{18}^3 - 
  3420 t_4^3 t_6 t_{10}^6 t_{12} \notag \\
  & - 107460 t_4^3 t_6 t_{10}^3 t_{12}^2 t_{18}  - 
  174960 t_4^3 t_6t_{12}^3 t_{18}^2  - 1889568  t_4^3 t_6 t_{18}^4 + 
  2772 t_4^4 t_6t_{10}^5 t_{18}  + 314928 t_4^4 t_6 t_{10}^2 t_{12} t_{18}^2 \notag \\
  & - 
  186624 t_4^5 t_6 t_{10} t_{18}^3  - 16 t_4^6 t_6 t_{10}^6  - 
  576 t_4^6 t_6 t_{10}^3 t_{12} t_{18}  - 3456 t_4^6 t_6 t_{12}^2 t_{18}^2  + 
  1152 t_4^7 t_6 t_{10}^2 t_{18}^2  - 5832 t_6^2 t_{10}^3 t_{12}^4  \notag \\
  &- 
  10125 t_6^2 t_{10}^6 t_{18}  - 157464 t_6^2 t_{12}^5 t_{18}  - 
  295245 t_6^2 t_{10}^3 t_{12} t_{18}^2  + 5314410 t_6^2 t_{12}^2 t_{18}^3  - 
  5670 t_4 t_6^2 t_{10}^5 t_{12}^2 \notag \\
  & - 170586 t_4 t_6^2 t_{10}^2 t_{12}^3 t_{18}  + 
  3188646 t_4 t_6^2  t_{10}^2 t_{18}^3 + 2700 t_4^2 t_6^2 t_{10}^7  + 
  101898 t_4^2 t_6^2  t_{10}^4 t_{12} t_{18} + 
  1102248 t_4^2 t_6^2 t_{10} t_{12}^2 t_{18}^2  \notag \\
  & + 216t_4^3 t_6^2 t_{10}^3 t_{12}^3  + 
  5832  t_4^3 t_6^2 t_{12}^4 t_{18}  - 195048  t_4^3 t_6^2 t_{10}^3 t_{18}^2  + 
  216 t_4^4 t_6^2 t_{10}^5 t_{12}  + 6480  t_4^4 t_6^2 t_{10}^2 t_{12}^2 t_{18} \notag \\
  &- 
  216  t_4^5 t_6^2 t_{10}^4 t_{18} - 20736 t_4^5 t_6^2 t_{10} t_{12} t_{18}^2  + 
  20736 t_4^6 t_6^2 t_{18}^3  + 6075 t_6^3 t_{10}^6 t_{12}  + 
  219429  t_6^3 t_{10}^3 t_{12}^2 t_{18} \notag \\
  & + 1338444 t_6^3 t_{12}^3 t_{18}^2  + 
  4251528 t_6^3 t_{18}^4  + 1215 t_4 t_6^3  t_{10}^5 t_{18} - 
  393660 t_4 t_6^3 t_{10}^2 t_{12} t_{18}^2  - 1259712 t_4^2 t_6^3 t_{10} t_{18}^3  \notag \\
&  - 
  216 t_4^3 t_6^3 t_{10}^6  - 7776 t_4^3 t_6^3 t_{10}^3 t_{12} t_{18}  - 
  46656 t_4^3 t_6^3 t_{12}^2 t_{18}^2  + 15552  t_4^4 t_6^3  t_{10}^2 t_{18}^2+ 
  729 t_6^4 t_{10}^3 t_{12}^3  + 19683 t_6^4 t_{12}^4 t_{18}  \notag \\
  &- 
  8748 t_6^4 t_{10}^3 t_{18}^2  - 2834352 t_6^4 t_{12} t_{18}^3  + 
  729 t_4 t_6^4 t_{10}^5 t_{12}  + 21870  t_4 t_6^4 t_{10}^2 t_{12}^2 t_{18} - 
  729 t_4^2 t_6^4 t_{10}^4 t_{18} \notag \\
  &  - 69984 t_4^2 t_6^4 t_{10} t_{12} t_{18}^2  + 
  139968   t_4^3 t_6^4 t_{18}^3 - 729 t_6^5 t_{10}^6  - 
  26244  t_6^5 t_{10}^3 t_{12} t_{18} - 157464 t_6^5 t_{12}^2 t_{18}^2  \notag \\
  &+ 
  52488 t_4 t_6^5 t_{10}^2 t_{18}^2  + 314928 t_6^6 t_{18}^3.
   \end{align}

 \begin{rem}\label{RemDiscr}
 The above $d_{90}(t)$ is calculated by the discriminant of the right hand side of (\ref{S(t)}).
 So, this is the natural counterpart of the famous discriminant $g_2^3 - 27g_3^2 $ of the Weierstrass form (\ref{EllipticWeierstrass}).
 For detail, see \cite{N} Section 1.
 \end{rem}

 \begin{prop}(\cite{N}, Corollary 1.1) \label{CorNSTr}
 For a generic point $[t] \in \mathcal{T}$,
 the N\'eron-Severi lattice  ${\rm NS}(S([t] ))$ is given by the intersection matrix $M=U\oplus E_8 (-1) \oplus E_6 (-1)$
 and the transcendental lattice ${\rm Tr} (S([t]))$ is given by $A$ of (\ref{LatticeA}).
 \end{prop}

 We define 
 the period mapping as follows.
 There exists a system of basis $\{\Gamma_1,\cdots,\Gamma_6,\Gamma_7\cdots,\Gamma_{22}\}$ of the $K3$ lattice $L_{K3}$ such that $\{\Gamma_7,\cdots,\Gamma_{22}\}$ is a system of basis of the lattice $M$.
 Letting $\{\Delta_1,\cdots,\Delta_{22}\} $ be the system of  dual basis of it,
  $\{\Delta_1,\cdots,\Delta_6\}$ gives a system of basis of the lattice $A$.
 From Proposition \ref{CorNSTr},
 we can take a point $[t_0]\in\mathcal{T}$
 and an identification
 $\psi_0:H_2(S_0,\mathbb{Z}) \rightarrow L_{K3}$,
 where 
  $S_0=S([t_0])$,
  such that $\{\psi_0^{-1}(\Gamma_7),\cdots,\psi_0^{-1}(\Gamma_{22}) \}$ is a basis of ${\rm NS}(S_0)$.
 If we take a sufficiently small neighborhood $U$ of $[t_0]$ in $\mathcal{T}$,
 we have a topological trivialization $\nu:\{S([t])|[t]\in U \} \rightarrow S_0\times U$.
 Letting $\beta :S_0\times U \rightarrow S_0$ be the projection and setting $r=\beta \circ \nu,$
 $r'_{[t]}=r |_{S([t])}$ gives a $\mathcal{C}^\infty$-isomorphism of complex surfaces
 for $[t]\in U.$
Then, $\psi_{[t]} = \psi_0 \circ (r'_{[t]})_* : H_2(S([t]),\mathbb{Z}) \rightarrow L_{K3}$ gives an isometry, 
 which is called the $S$-marking in \cite{N}.
 By an analytic continuation along an arc in $ \mathcal{T}$,
 we can define the $S$-marking on $\mathcal{T}$.
 The period mapping $\Phi: \mathcal{T}\rightarrow \mathcal{D}$ is defined by
 \begin{align}\label{PerPhi}
  [t] \mapsto \Big(\int_{\psi_{[t]}^{-1} (\Gamma_1)} \omega_{[t]} : \cdots: \int_{\psi_{[t]}^{-1} (\Gamma_6)} \omega_{[t]} \Big).
 \end{align}
In Section 2, we will give a geometric construction of these period integrals.
The mapping (\ref{PerPhi}) is an multivalued mapping.
We can prove that this mapping  induces the biholomorphic isomorphism 
\begin{align}\label{PhiIso}
\bar{\Phi} : T \simeq \mathcal{D}/\Gamma
\end{align}
by a detailed argument of lattice polarized $K3$ surfaces
(see \cite{N} Section 2).

 Let $\mathcal{D}^*$ be the $\mathbb{C}^\times$-bundle of $\mathcal{D}$.
 Due to the Kneser condition of the lattice $A$, 
 it follows that
 the group $\Gamma = \tilde{O}^+ (A)$ is generated by the reflections of $(-2)$-vectors
 and ${\rm Char}(\Gamma)={\rm Hom}(\Gamma, \mathbb{C}^\times)$ is equal to $\{{\rm id}, {\rm det}\}$.
 Moreover, we can investigate the orbifold structure of $[\mathcal{D}/\Gamma]$ as follows.

\begin{prop}(\cite{N}, Section 5)\label{PropReflection}
(1) 
 There exist two images of  reflection hyperplanes of $(-2)$-vectors  of the lattice $A$ under  the natural surjection $\mathcal{D} \rightarrow \mathcal{D}/\Gamma$.

(2) Under the isomorphism $\bar{\Phi}$ of (\ref{PhiIso}), the branch loci of the orbifold $[\mathcal{D}/\Gamma ]$ is corresponding to 
the union of the locus $\{t_{18}=0\}$ and $\{d_{90}(t) =0\}$.
Each locus is corresponding to each orbit of the reflection hyperplanes.
On $\{t_{18}=0\}$ ($\{d_{90} (t) =0\}$, resp.), the N\'eron-Severi lattice of $S([t])$ is degenerated to $U\oplus E_8(-1) \oplus E_7(-1)$
($U\oplus E_8(-1) \oplus E_6(-1) \oplus A_1(-1)$, resp.).
\end{prop}

 \begin{df}\label{ModularIV}
 If a holomorphic function $f:\mathcal{D}^* \rightarrow \mathbb{C}$ 
given by $Z\mapsto f(Z)$ satisfies the conditions
\begin{itemize}

\item[(i)] $f(\lambda Z) = \lambda^{-k} f(Z) \quad (\text{for all } \lambda \in \mathbb{C}^*)$,

\item[(ii)] $f(\gamma Z) = \chi(\gamma) f(Z) \quad (\text{for all } \gamma \in \Gamma),$

\end{itemize}
where $k\in \mathbb{Z}$ and $\chi \in {\rm Char}(\Gamma)$,
then $f$ is called a modular form of weight $k$ and character $\chi$ for the group $\Gamma$.
Let $\mathcal{A}(\Gamma)$ be the ring of such modular forms.
\end{df}

\begin{thm}\label{ThmModular} (\cite{N}, Theorem 5.1)
(1)  
The ring  $\mathcal{A}(\Gamma,{\rm id})$ of modular forms of  character {\rm id} is isomorphic to
the ring  $\mathbb{C}[t_4,t_6,t_{10},t_{12},t_{18}]$.
Namely, 
via the inverse of the period mapping $\bar{\Phi}$ of (\ref{PhiIso}),
the correspondence $Z\mapsto t_k(Z)$ gives a modular form of weight $k$ and character ${\rm id}$.

(2) There is a modular form $s_{54}$  of weight $54$ and character ${\rm det}$.
Here, $s_{54}$ is given by $s_9 s_{45}$, where
$s_9$ and $s_{45}$ are holomorphic functions on $\mathcal{D}^*$ such that
\begin{align*}
s_9^2 =t_{18},  \quad \quad
s_{45}^2 =d_{90}(t).
\end{align*}
These relations determine the structure of the ring $\mathcal{A}(\Gamma)$.
\end{thm}

By the way,
Clingher-Doran \cite{CD} studied a family of $K3$ surface
\begin{align}\label{CDFamily}
S_{CD}(\alpha:\beta:\gamma:\delta): 
z^2 =y^3+ (-3\alpha x^4-\gamma x^5) y+ (x^5-2\beta x^6 +\delta x^7) 
\end{align}
 parametrized by $(\alpha:\beta:\gamma:\delta)\in \mathbb{P}(2,3,5,6)-\{\gamma=\delta=0\}$. 
This family is a subfamily of our family $\{S([t])\}$ as follows.

\begin{prop} (\cite{N}, Proposition 2.3) \label{PropS(t)CD}
 The family  $\{S_{CD} (\alpha,\beta,\gamma,\delta) \}$  is equal to the subfamily  of $\{S([t])\}$  given by the condition $t_{18}=0.$
 Especially, the embedding $(\alpha:\beta:\gamma:\delta) \hookrightarrow (t_4:t_6:t_{10}:t_{12}:t_{18})$
 of parameters 
  is given by 
 $$t_4=-3\alpha,\quad  t_6=-2\beta, \quad  t_{10}=-\gamma, \quad  t_{12}=\delta, \quad t_{18}=0.$$  
 \end{prop}

The inverse period mapping for the Clingher-Doran family is closely related to Siegel modular forms of degree $2$ (see  Proposition \ref{PropCDIgusa}).
We note that this family is a simple extension of the famous family of \cite{SI}.
Also, this family contains the family of \cite{NHilb},
which is closely related to the Hilbert modular forms and the icosahedral invariants (for detail, see \cite{NS} Theorem 5.4).

\section{Geometric construction of periods}

In this section, 
we will give an explicit geometric  construction of  $2$-cycles  on a reference surface $S_0$ in Section 1. 
Together with analytic continuation,
this also gives a construction of the period integrals  of (\ref{PerPhi}).
Our construction is based on the geometry of elliptic surfaces,
 like the argument of  \cite{NS} Section 8.

First, let $[t_0] = (t_4: t_6: t_{10}: t_{12}: t_{18})=(-12: -3: -5: 21:13)\in \mathcal{T}$.
Let us give a geometric  construction of $2$-cycles 
on
 a reference surface $S_0=S([t_0])$.
By putting $w=1$ and performing the birational transformation
$$
x=\frac{1}{x_1}, \quad y=\frac{y_1}{x_1^4}, \quad z=\frac{z_1}{x_1^6} ,
$$
the equation (\ref{S(t)}) of our $K3$ surface  is transformed to
\begin{eqnarray}\label{S(t)1}
&&
 z_1^2=y_1^3+(-12x_1^4-5 x_1^5)y_1+(x_1^5-4x_1^6+21x_1^7+13x_1^8).
\end{eqnarray}
The equation (\ref{S(t)1}) defines an elliptic fibration $\pi_1:(x_1,y_1,z_1)\mapsto x_1$.
On $x_1=0$ ($x_1=\infty$, resp.), we have a singular fibre of $\pi_1$ of Kodaira type $II^*$ ($IV^*$, resp.).  
Set $\alpha_0=0$ and $\alpha_{\infty} =\infty$.
Furthermore, 
we have other $6$ singular fibres $\pi_1^{-1} (\alpha_i)$ $(i=1,\cdots,6)$ of Kodaira type $I_1$,
where
$$
(\alpha_1,\alpha_2,\alpha_3,\alpha_4,\alpha_5,\alpha_6) 
\approx
 (-1.84  , -1.65  , -0.43 , 
-0.10  , 0.05  , 0.84  ).
$$
We call $\alpha_j\in \mathbb{R}\cup \{\infty\}$ $(j=0,1,\cdots,6,\infty)$ the critical points.
Take
$b_0=i \in x_1$-plane.
We call this point the base point.
The fibre $\pi_1^{-1}(b_0)$ is given by the elliptic curve
$$
z_1^2=  y_1^3 - (12 + 5 i) y_1 + (16 - 20 i).
$$
This defines a double covering of the $y_1$-sphere with $4$ branch points 
$$
c_1=-4 ,\quad  c_2\approx 0.418861 - 1.58114 i,  \quad  c_3\approx 3.58114 + 1.58114 i, \quad c_\infty=\infty. 
$$
Take a basis $\{\gamma_1,\gamma_2\}$ of $H_1(\pi_1^{-1}(b_0),\mathbb{Z})$ as in Figure 1.
The intersection number $\gamma_1\cdot \gamma_2$  is equal to $1$.
\begin{figure}[h]
\center
\includegraphics[scale=0.7]{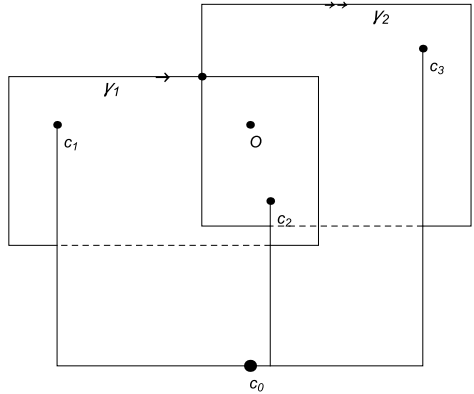}
\caption{$1$-cycles $\gamma_1$ and $\gamma_2$ on the elliptic curve $\pi^{-1}(b_0)$}
\end{figure}
Let $\delta_i$ $(i=0,1,\cdots,6,\infty)$ be a closed oriented loop in $x_1$-plane 
starting at the
base point $b_0$ 
and going around $\alpha_i$ in the positive direction.
We call them  circuits (see Figure 2).
\begin{figure}[h]
\center
\includegraphics[scale=0.8]{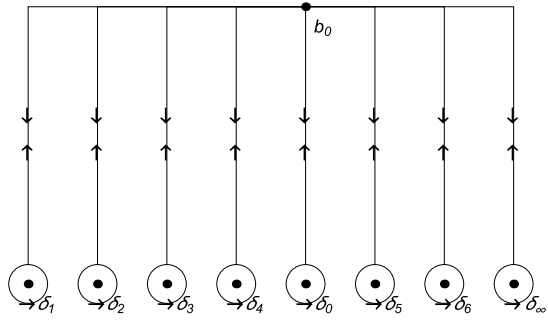}
\caption{Circuits $\delta_j$ going around critical points }
\end{figure}
The circuit $\delta_i$  induces a monodromy transformation of the system $\{\gamma_1,\gamma_2\}$. 
Let us denote it by the matrix $M_i$ as a left action. We call them local monodromies.

\begin{lem}\label{LemMonodromy}
The local monodromies are given by Table 2.
\end{lem}

\begin{table}[h]
\center
{\small
\begin{tabular}{ccccccccccc}
\hline
Points&$\alpha_1$&$\alpha_2$&$\alpha_3$&$\alpha_4$&$\alpha_0=0$
&$\alpha_5$&$\alpha_6$&$\alpha_{\infty} =\infty$\\ \hline
$M_i$ &$\begin{pmatrix}0&-1\cr 1&2\end{pmatrix}$&
$\begin{pmatrix}1&-1\cr 0&1\end{pmatrix}$&
$\begin{pmatrix}1&-1\cr 0&1\end{pmatrix}$&
$\begin{pmatrix}1&-1\cr 0&1\end{pmatrix}$&
$\begin{pmatrix}0&1\cr -1&1\end{pmatrix}$&
$\begin{pmatrix}1&0\cr 1&1\end{pmatrix}$&
$\begin{pmatrix}1&0\cr 1&1\end{pmatrix}$&
$\begin{pmatrix}0&1\cr -1&-1\end{pmatrix}$\\ \hline
Type&$I_1$&$I_1$&$I_1$&$I_1$&$II^{\ast}$&$I_1$&$I_1$&$IV^{\ast}$\\ \hline
V. Cycle&$\gamma_1+\gamma_2$&$\gamma_2$&$\gamma_2$&$\gamma_2$&-&$\gamma_1$&$\gamma_1$&-\\ \hline
\hline 
\end{tabular}
}
\caption{Local monodromies around $\alpha_i$}
\end{table}

\begin{rem}
In Table 2,
``V. Cycle'' means the vanishing cycle of the elliptic fibre at the critical point $\alpha_i$.
\end{rem}

Let $\delta$ be 
an oriented arc in the $x_1$-plane
 starting at
 the base point $b_0$
 and let $\gamma\in H_1(\pi_1^{-1} (b_0),\mathbb{Z})$. 
We obtain a $2$-chain $\Gamma(\delta,\gamma)$ on the reference surface (\ref{S(t)1})
obtained by the continuation of $\gamma$ along $\delta$.
We define the orientation of  $\Gamma(\delta,\gamma)$ by the ordered pair of $\delta$ and $\gamma$.
If $\delta$ is a loop returning back to the starting cycle $\gamma$,  
then $\Gamma(\delta,\gamma)$ becomes to be a $2$-cycle on the reference surface.
By using the local monodromies in Lemma \ref{LemMonodromy},
we have the following $2$-cycles.
\begin{eqnarray*}
&&
G_0^{\ast} =\Gamma(\delta_1, \gamma_2)+ \Gamma(\delta_2, \gamma_1+2\gamma_2)+ \Gamma(-\delta_6, \gamma_1+\gamma_2),\\
&&
G_1^{\ast} =\Gamma(\delta_2, \gamma_1)+\Gamma(-\delta_3, \gamma_1-\gamma_2),
\\&&
G_2^{\ast} =\Gamma(\delta_3, \gamma_1)+ \Gamma(-\delta_4, \gamma_1-\gamma_2),
\\&&
G_3^{\ast} =\Gamma(\delta_4, \gamma_1)+ \Gamma(\delta_0, \gamma_1-\gamma_2),
\\&&
G_4^{\ast} =\Gamma(\delta_0, \gamma_2)+ \Gamma(\delta_5, -\gamma_1+\gamma_2),
\\&&
G_5^{\ast} =\Gamma(\delta_5, \gamma_2)+ \Gamma(-\delta_6, -\gamma_1+\gamma_2).
\end{eqnarray*}
They are illustrated in Figure 3.
Here, the line
$l_i=\{(\alpha_i, -it) | t\geq 0\}$
is called the cut line.
We note that
the circuit $\delta_i$ meet the cut line $l_i$ just  once.
The system $\{ G_0^{\ast}, G_1^{\ast}, \cdots , G_5^{\ast}\}$ has the intersection matrix
\begin{eqnarray*}
&&
\begin{pmatrix}
-2& -1& 0& 0& 0& -1\cr -1&-2& 2& 0& 0& 0\cr 0&2& -2& 1& 0& 0\cr 
0&0& 1& 0& 0& 0\cr 0& 0& 0& 0& 0& -1\cr -1& 0& 0& 0& -1& -2
\end{pmatrix}
.
\end{eqnarray*}
\begin{figure}[h]
\center
\includegraphics[scale=1]{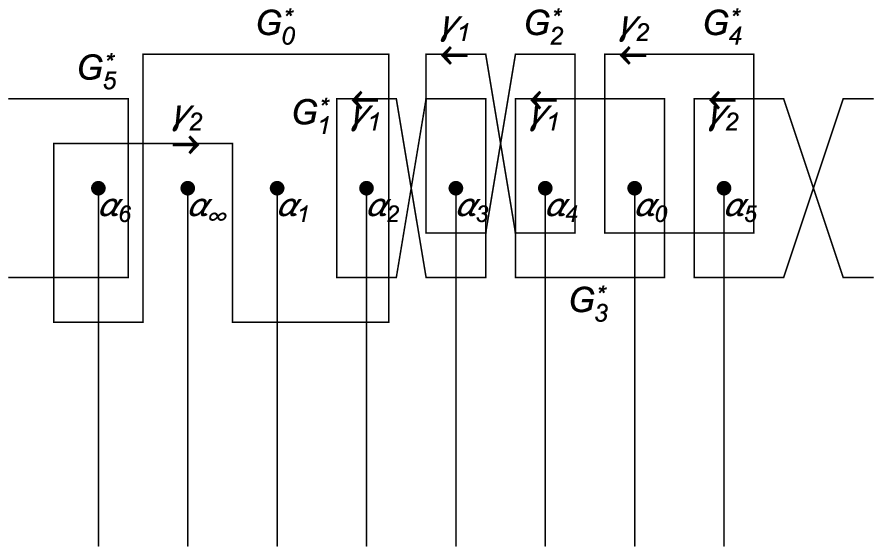}
\caption{2-cycles $G_i^{\ast}$ }
\end{figure}
We set
\begin{eqnarray*}
&&
T_n=
\begin{pmatrix}
-1& 0& 0& 0& 1& 0\cr 0& 1& 0& -2& 0& 0\cr 0&0& 1& 1& 0& 0\cr 
0& 0&  0& 1& 0& 0\cr 0&0&0& 0& 1& 0\cr 0& 0& 0& 0& 1& -1
  \end{pmatrix}
\end{eqnarray*}
and
$$
{}^t (H_0^*,H_1^*,H_2^*,H_3^*,H_4^*,H_5^*) =T_n  {}^t (G_0^*,G_1^*,G_2^*,G_3^*,G_4^*,G_5^*).
$$
We can see that
the intersection matrix of the system $\{H_0^*,H_1^*,\cdots,H_5^*\}$ is equal to
$
A_2(-1)\oplus U\oplus U.
$
Now,
we remark that 
we can take a system of basis $\{C_1,\cdots,C_{16}\}$ of ${\rm NS}(S_0)$
such that each of $\{C_j\}_{j}$ are contained in $\pi_1^{-1}(\alpha_0)\cup \pi_1^{-1}(\alpha_\infty) $. 
Note that $H_0^*,\cdots , H_5^*$ do not meet the singular fibres of $\pi_1$.
Therefore, we have the following result.

\begin{prop}\label{PropHSystem}
The set $\{H_0^*,\cdots, H_5^*\}$ of  $2$-cycles  gives a basis of ${\rm Tr}(S_0)$.
\end{prop}

Next, let us construct another system $\{G_0,G_1,\cdots,G_5 \}$ of $2$-cycles.
Let $\varrho_i$ $(i=0,1,\cdots,5,\infty)$ be an arc in $x_1$-plane
whose start point and end point are given by Table 3 (see also Figure 4).
\begin{table}[h]
\center
\begin{tabular}{ccccccccccc}
\toprule
 & $\varrho_0$ &   $\varrho_1$    &   $\varrho_2$  & $\varrho_3$ &  $\varrho_4$ &  $\varrho_5$    \\
 \midrule
start point & $\alpha_\infty$  & $\alpha_\infty$  & $\alpha_0$ & $\alpha_\infty$ & $\alpha_0$ & $\alpha_\infty$ \\
end point  & $\alpha_1       $ & $\alpha_2$         & $\alpha_3$ & $\alpha_0$       & $\alpha_4$ & $\alpha_6$ \\
 \bottomrule
\end{tabular}
\caption{Arcs $\varrho_i$}
\end{table} 
Recalling the vanishing cycles in Table 2,
we have the following  $2$-cycles
\begin{align*}
& G_0=\Gamma(\varrho_0,\gamma_1+\gamma_2),\quad G_1=\Gamma(\varrho_1,-\gamma_2),\quad G_2=\Gamma(\varrho_2,\gamma_2),\\
&G_3=\Gamma(\varrho_3,\gamma_1),\quad G_4=\Gamma(\varrho_4,\gamma_2),\quad
G_5=\Gamma(\varrho_5,\gamma_1).
\end{align*}
\begin{figure}[h]
\center
\includegraphics[scale=0.92]{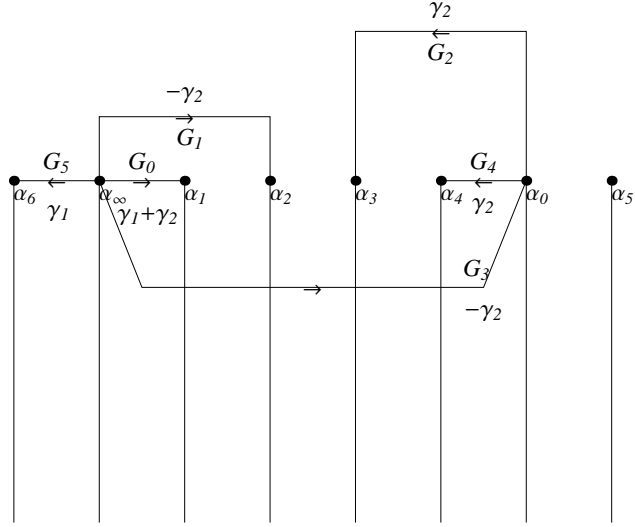}
\caption{$2$-cycles $G_j$ }
\end{figure}
Set
\begin{eqnarray*}
&&
S_n=\begin{pmatrix}
-1& 0& 0& 0& 0& 0\cr 0& 1& 0& 0& 0& 0\cr 0&1&- 1& 1& 0& 0\cr 
0& 1& 1& 0& 0& 0\cr 1&1&-1& 1& -1& 1\cr 0& 0& 0& 0& 0& -1
  \end{pmatrix}
\end{eqnarray*}
and
$$
{}^t (H_0,H_1,H_2,H_3,H_4,H_5)=S_n {}^t (G_0,G_1,G_2,G_3,G_4,G_5).
$$
From Figure 3 and 4,
the intersection numbers for the system $\{H_i\}_i$ and $\{H_i^*\}_i$
are calculated as
\begin{eqnarray}\label{HDual}
&&
H_i \cdot H^{\ast}_j =\delta_{ij}  \quad  (0\leq i, j \leq 5).
\end{eqnarray}
Especially,
from Proposition \ref{PropHSystem} and (\ref{HDual}),
$\{H_0,H_1,\cdots,H_5,C_1,\cdots,C_{16}\}$ gives a basis of $H_2(S_0,\mathbb{Z})$.
Recalling the construction of the period mapping $\Phi$ of (\ref{PerPhi}), 
we have the following theorem.

\begin{thm}\label{ThmGeomCycle}
There exists an $S$-marking $\psi_{[t]}:H_2(S([t]),\mathbb{Z}) \rightarrow L_{K3}$ in the sense of Section 1
so that
the initial marking
$\psi_{0}:H_2(S_0,\mathbb{Z})\rightarrow L_{K3}$
 satisfies
\begin{align*}
& \psi_{0}(H_2)=\Gamma_1,\quad \psi_{0}(H_3)=\Gamma_2, \quad \psi_{0}(H_4)=\Gamma_3, \quad \psi_{0}(H_5)=\Gamma_4,\quad \psi_{0}(H_0)=\Gamma_5, \quad \psi_{0}(H_1)=\Gamma_6,\\
& \psi_{0} (C_j) =\Gamma_{j+6} \quad (j\in \{1,\cdots,16\}).
\end{align*}
Especially, there is a branch of the multivalued mapping $\Phi$ of (\ref{PerPhi}) with the special value
$$
\Phi([t_0]) = \Big( \int_{H_2} \omega_{[t_0]} :\int_{H_3} \omega_{[t_0]} :\int_{H_4} \omega_{[t_0]} : \int_{H_5} \omega_{[t_0]} :\int_{H_0} \omega_{[t_0]} :\int_{H_1} \omega_{[t_0]} \Big).
$$
\end{thm}

\section{Hermitian modular forms}

\subsection{Definition of Hermitian modular forms}

Set
\begin{align}\label{HI}
\mathbb{H}_I 
=\Big\{ W=\begin{pmatrix} \tau & z \\ w & \tau ' \end{pmatrix} \in M_2(\mathbb{C}) 
\Big| 
\frac{1}{2 \sqrt{-1} } 
(W-{}^t \overline{W} ) >0  
\Big\}.
\end{align}
This is a $4$-dimensional complex bounded symmetric domain of type $I$.

Let $K$ be an imaginary quadratic field.
Let $\mathfrak{O}_K$ be its ring of integers.
Letting $J=\begin{pmatrix} 0 & -I_2 \\ I_2 &0 \end{pmatrix}$,
set 
\begin{align}
\Gamma(\mathfrak{O}_K)=\Big\{ M\in M_4(\mathfrak{O}_K) \Big| M J {}^t\overline{M}=J \Big\}/\{\pm 1\}.
\end{align}
This group plays a role as a full-modular group for the Hermitian modular forms.
An element $M=\begin{pmatrix}A & B \\ C & D \end{pmatrix}\in\Gamma(\mathfrak{O}_K)$ acts on $\mathbb{H}_I$ by
$$
W \mapsto (AW +B)(CW+D)^{-1}.
$$

\begin{df}\label{ModularI}
Let $k\in \mathbb{Z}$ and $\nu\in {\rm Hom}(\Gamma(\mathfrak{O}_K),\mathbb{C}^\times)$.
A Hermitian modular form of weight $k$ and character $\nu$ for the group $\Gamma(\mathfrak{O}_K)$ is 
a holomorphic function $f$ on $\mathbb{H}_I$
satisfying
\begin{align*}
(f|_k M) (W) = \nu(M)  f(W),
\end{align*}
where
\begin{align}\label{ModularFormTrans}
(f|_k M) (W):= {\rm det}(CW+D)^{-k} f((AW+B)(CW+D)^{-1}).
\end{align}
\end{df}

By specialists in modular forms, 
 structures of rings of Hermitian modular forms are determined in several cases.
For example, 
when $K=\mathbb{Q}(\sqrt{-3})$,
Dern and Krieg \cite{DK} studied the ring of Hermitian modular forms in detail by applying techniques of Eisenstein series and Borcherds products.

\begin{prop}  \label{PropDK}
(1) (\cite{DK}, Theorem 7)
There are algebraically independent Hermitian modular form $E_4$ ($E_6,\phi_9,E_{10}, E_{12}$, resp.) of weight $4$ ($6,9,10,12$, resp.) for the case $K=\mathbb{Q}(\sqrt{-3}).$
There is a modular form $\phi_{45}$ of weight $45$ and a polynomial $p(X_1,X_2,X_3,X_4,X_5)$ such that 
$\phi_{45}^2 = p(E_4,E_6,\phi_9,E_{10},E_{12}).$

(2) (\cite{DK}, Corollary 3)
Let $W=\begin{pmatrix} \tau & z \\ w & \tau' \end{pmatrix}\in \mathbb{H}_I$.
Then, the modular form $\phi_9$ ($\phi_{45}$, resp.) is identically equal to $0$ on the locus $\{z=w\}$ ($\{z=-w\}$, resp.).
\end{prop}

\subsection{Modular isomorphism and moduli of $K3$ surfaces}

In our work, we will consider the case of $K=\mathbb{Q}(\sqrt{-3})$.

We have the following modular isomorphism
\begin{align}\label{ModularIso}
f: \mathcal{D}& \ni Z= {}^t (\xi_1:\xi_2:\xi_3:\xi_4:\xi_5:\xi_6) \notag \\
&\mapsto W= \frac{1}{\xi_1}\begin{pmatrix} \xi_3 & \frac{1+\sqrt{-3}}{2} \xi_5 +\frac{1-\sqrt{-3}}{2} \xi_6 \\ \frac{1-\sqrt{-3}}{2} \xi_5 + \frac{1+\sqrt{-3}}{2} \xi_6 & \xi_4 \end{pmatrix}\in \mathbb{H}_I.
\end{align}
This induces an isomorphism $\tilde{f}: \tilde{O}^+(A) \simeq \langle \Gamma(\mathfrak{O}_K),  T_1 \rangle $ of modular groups,
where $T_1$ is the involution given by $W\mapsto {}^t W$. 
Also, we need the involution
$T_2$ which is the involution given by $W=\begin{pmatrix} \tau & z \\ w & \tau' \end{pmatrix}  \mapsto \begin{pmatrix} \tau & -w \\ -z & \tau' \end{pmatrix}. $
We remark  $T_2\in \langle \Gamma(\mathfrak{O}_K),T_1\rangle$.

Let 
$S_1=\begin{pmatrix} 1 & 0 \\ 0 & 0 \end{pmatrix}, S_2 =\begin{pmatrix} 0 & 0 \\ 0 & 1 \end{pmatrix}$ and
$S_3=\begin{pmatrix} 0 & 1 \\ 1 & 0 \end{pmatrix}$.
We  need the actions  given by the matrices
\begin{align}\label{Action}
\begin{cases}
& M_j=\begin{pmatrix} I_2 & S_j \\ 0 & I_2  \end{pmatrix}\in \Gamma(\mathfrak{O}_K) \quad \text{ for }
S_j,\\
& J=\begin{pmatrix} 0 & -I_2 \\ I_2 &0 \end{pmatrix} \in \Gamma(\mathfrak{O}_K).
\end{cases}
\end{align}

\begin{rem}
Recalling Lemma \ref{LemStableG},
we have
\begin{align*}
& \tilde{f}^{-1} (M_1)
=\begin{pmatrix} 
1& 0 & 0& 0& 0 &0 \\
0& 1 & 0& -1& 0 &0 \\
1& 0 & 1& 0& 0 &0 \\
0& 0 & 0& 1& 0 &0 \\
0& 0 & 0& 0& 1 &0 \\
0& 0 & 0& 0& 0 &1 \\
  \end{pmatrix},
  \quad
  \tilde{f}^{-1}(M_2)
  =\begin{pmatrix} 
1& 0 & 0& 0& 0 &0 \\
0& 1 & -1& 0& 0 &0 \\
0& 0 & 1& 0& 0 &0 \\
1& 0 & 0& 1& 0 &0 \\
0& 0 & 0& 0& 1 &0 \\
0& 0 & 0& 0& 0 &1 \\
  \end{pmatrix},\\
  &\tilde{f}^{-1}(M_3)
  =\begin{pmatrix} 
1& 0 & 0& 0& 0 &0 \\
1& 1 & 0& 0& 1 &1 \\
0& 0 & 1& 0& 0 &0 \\
0& 0 & 0& 1& 0 &0 \\
1& 0 & 0& 0& 1 &0 \\
1& 0 & 0& 0& 0 &1 \\
  \end{pmatrix},
  \quad 
   \tilde{f}^{-1}(J)
  =\begin{pmatrix} 
0&-1 & 0& 0& 0 &0 \\
-1& 0 & 0& 0& 0 &0 \\
0& 0 & 0& -1& 0 &0 \\
0& 0 & -1& 0& 0 &0 \\
0& 0 & 0& 0& 1 &0 \\
0& 0 & 0& 0& 0 & 1 \\
  \end{pmatrix},\\
 &\tilde{f}^{-1}(T_1)
  =\begin{pmatrix} 
-1& 0 & 0& 0& 0 &0 \\
0& -1 & 0& 0& 0 &0 \\
0& 0 & -1& 0& 0 &0 \\
0& 0 & 0& -1& 0 &0 \\
0& 0 & 0& 0& 0 &-1 \\
0& 0 & 0& 0& -1 & 0 \\
  \end{pmatrix}, 
  \quad
  \tilde{f}^{-1}(T_2)
  =\begin{pmatrix} 
1& 0 & 0& 0& 0 &0 \\
0& 1 & 0& 0& 0 &0 \\
0& 0 & 1&0& 0 &0 \\
0& 0 & 0& 1& 0 &0 \\
0& 0 & 0& 0& 0 &-1 \\
0& 0 & 0& 0& -1 & 0 \\
  \end{pmatrix}.
\end{align*}
They are elements of $\tilde{O}^+(A).$
\end{rem}

Under the modular isomorphism $f$ of (\ref{ModularIso}),
the modular forms  in the sense of Definition \ref{ModularIV}
are regarded as
the Hermitian modular forms on $\mathbb{H}_I$  of Definition \ref{ModularI}.
Due to this identification, 
Theorem \ref{ThmModular} is concordant with Proposition \ref{PropDK}.
In particular, we have the following result.

\begin{prop} \label{PropModularFormIso}
The modular form $t_{18}(Z)$ ($d_{90}(Z)$, resp.) in the sense of Theorem \ref{ThmModular}
vanishes on the locus $f^{-1}(\{z=w\})$ ($f^{-1}(\{z=-w$\}), resp.), where $f$ is the modular isomorphism (\ref{ModularIso}).
\end{prop}

This means that
we have a natural geometric meaning of the Hermitian modular forms 
using the period mapping of $K3$ surfaces.
Especially,
in \cite{DK}, it is  proved the existence of the polynomial $p(X_1,X_2,X_3,X_4,X_5)$ in Proposition \ref{PropDK} based on techniques of modular forms.
Due to  Proposition \ref{PropReflection}, Theorem \ref{ThmModular} and Remark \ref{RemDiscr},
which are results of  $K3$ surfaces,
we are able to  give a natural geometric meaning to  $p(X_1,X_2,X_3,X_4,X_5)$ by identifying this  with $d_{90}(t)$ of (\ref{d90}).

\section{Theta expression of  inverse period mapping via  Burkhardt invariants}

In this section,
we will obtain the complete theta expression of the inverse period mapping of $S([t])$ in (\ref{S(t)}).

In the argument below,
$O(x,y)$ means a power series in $x$ and $y$ 
with the constant term $0$.

\subsection{Igusa invariants }

Let us review the theta functions on the Siegel upper half plane $\mathfrak{S}_2$ of degree $2$.
For $W_0\in \mathfrak{S}_2$ and 
$$ 
\displaystyle m_j=\begin{pmatrix} s \\ t \end{pmatrix}, 
\quad
n_j=\begin{pmatrix} u \\ v \end{pmatrix},
$$
set
\begin{align}\label{SiegelTheta}
\vartheta_j (W_0)=\sum_{(a, b)\in \mathbb{Z}^2} e^{\pi \sqrt{-1} (a+\frac{1}{2}s , b+\frac{1}{2} t) W_0  {}^t(a+\frac{1}{2}s , b+\frac{1}{2} t)  +2\pi \sqrt{-1} (a+\frac{1}{2}s , b+\frac{1}{2} t) {}^t(\frac{1}{2}u, \frac{1}{2} v )  }
\end{align}
for $j\in \{0,1,\cdots, 9\}$.
Here, the correspondence between $j $ and the characteristics $(m_j,n_j)$ is given by Table 4.

\begin{table}[h]
\center
\begin{tabular}{ccccccccccc}
\toprule
$j$ & $0$ &   $1$    &   $2$  & $3$ &  $4$ &  $5$ &   $6  $ &  $7$   &  $8$ & $9$ \\
 \midrule
$m_j$ & $\begin{pmatrix} 0 \\ 0  \end{pmatrix} $ & $\begin{pmatrix} 1 \\ 0  \end{pmatrix}$ & $\begin{pmatrix} 0 \\ 1  \end{pmatrix} $ &$\begin{pmatrix} 1 \\ 1  \end{pmatrix} $ & $\begin{pmatrix} 0 \\ 0  \end{pmatrix} $ & $\begin{pmatrix} 0 \\ 0  \end{pmatrix} $ & $\begin{pmatrix} 0 \\ 0  \end{pmatrix} $ & $\begin{pmatrix} 1 \\ 0  \end{pmatrix} $ & $\begin{pmatrix} 0 \\ 1  \end{pmatrix} $ & $\begin{pmatrix} 1 \\ 1  \end{pmatrix} $ \\
$n_j$ & $\begin{pmatrix} 0 \\ 0  \end{pmatrix}$  &$\begin{pmatrix} 0 \\ 0  \end{pmatrix} $ & $\begin{pmatrix} 0 \\ 0  \end{pmatrix}$ & $\begin{pmatrix} 0 \\ 0  \end{pmatrix} $ & $\begin{pmatrix} 1 \\ 0  \end{pmatrix}$ & $\begin{pmatrix} 0 \\ 1  \end{pmatrix}$ & $\begin{pmatrix} 1 \\ 1  \end{pmatrix}$ & $\begin{pmatrix} 0 \\ 1  \end{pmatrix}$ & $\begin{pmatrix} 1 \\ 0  \end{pmatrix}$ & $\begin{pmatrix} 1 \\ 1  \end{pmatrix}$  \\
 \bottomrule
\end{tabular}
\caption{Correspondence between $j$ and  $(r_j , s_j)$}
\end{table} 

For 
$
W_0=\begin{pmatrix} \tau & z \\ z & \tau' \end{pmatrix} \in \mathfrak{S}_2,
$
we set
\begin{align}\label{q}
q_1=e^{2\pi \sqrt{-1} \tau}, \quad q_2=e^{2\pi \sqrt{-1} \tau'}, \quad \zeta=e^{2\pi \sqrt{-1}z}, \quad
\tilde{q_1}=e^{\pi \sqrt{-1} \tau},\quad \tilde{q_2}=e^{\pi \sqrt{-1} \tau'}.
\end{align}
We have the following Fourier expansion of them.

\begin{prop}\label{PropIgusaTheta}
\begin{align*}
& \vartheta_0(W_0)=1 +O(\tilde{q_1},\tilde{q_2}),\\
& \vartheta_1(W_0)=q_1^{\frac{1}{8}}(2 +O(\tilde{q_1},\tilde{q_2})),\\
& \vartheta_2(W_0)=q_2^{\frac{1}{8}}(2 +O(\tilde{q_1},\tilde{q_2})),\\
& \vartheta_3(W_0)=q_1^{\frac{1}{8}}q_2^{\frac{1}{8}}(2(\zeta^{\frac{1}{4}} + \zeta^{-\frac{1}{4}}) +O(\tilde{q_1},\tilde{q_2})),\\
& \vartheta_4(W_0)=1 +O(\tilde{q_1},\tilde{q_2}),\\
& \vartheta_5(W_0)=1 +O(\tilde{q_1},\tilde{q_2}),\\
& \vartheta_6(W_0)=1 +O(\tilde{q_1},\tilde{q_2}),\\
& \vartheta_7(W_0)=q_1^{\frac{1}{8}}(2 +O(\tilde{q_1},\tilde{q_2})),\\
& \vartheta_8(W_0)=q_2^{\frac{1}{8}}(2 +O(\tilde{q_1},\tilde{q_2})),\\
& \vartheta_9(W_0)=q_1^{\frac{1}{8}}q_2^{\frac{1}{8}}(2(-\zeta^{\frac{1}{4}} + \zeta^{-\frac{1}{4}}) +O(\tilde{q_1},\tilde{q_2})).\\
\end{align*}
\end{prop}

\begin{proof}
$\vartheta_j (W_0)$ has an expression 
$$
\vartheta_j(W_0) =q_1^{\alpha_1} q_2^{\alpha_2} \sum_{(a,b)\in \mathbb{Z}^2} e^{2\pi\sqrt{-1} (\varphi_1(a,b) \tau +\varphi_2(a,b)\tau')}  e^{2\pi \sqrt{-1}( \psi_1(a,b) z+ \psi_2 (a,b)w)},
$$
where $\varphi_l(a,b),\psi_l(a,b) \in \mathbb{C}[a,b,c,d]$ $(l\in\{1,2\})$ and  $\varphi_l(a,b)$ does not contain any constant terms.
For each $j$,  
the pair $(a,b)$ satisfying $\varphi_1(a,b)=\varphi_2(a,b)=0 $ are given in Table 5.
From that table, the assertion follows.

\begin{table}[h]
\center
\begin{tabular}{ccccccccccc}
\toprule
$k$ & $\alpha_1$ & $\alpha_2$ & $(a,b)$ for $\varphi_1(a,b)=\varphi_2(a,b)=0 $  \\
 \midrule
 $0$ & $0$ & $0$ & $(0,0)$  \\
$1$ & $\frac{1}{8}$  & $0$ & $(0,0),(-1,0)$  \\
 $2$ & $0$ & $\frac{1}{8}$ & $(0,0),(0,-1)$ \\
 $3$ & $\frac{1}{8}$  & $\frac{1}{8}$ & $(0,0),(-1,0),(0,-1),(-1,-1)$   \\
 $4$ & $0 $  & $0$ & $(0,0)$ \\
 $5$ & $0 $  & $0$ & $(0,0)$ \\
 $6$ & $0 $  & $0$ & $(0,0)$ \\
 $7$ & $\frac{1}{8}$  & $0$ & $(0,0),(-1,0)$  \\
 $8$ & $0$ & $\frac{1}{8}$ & $(0,0),(0,-1)$ \\
 $9$ & $\frac{1}{8}$  & $\frac{1}{8}$ & $(0,0),(-1,0),(0,-1),(-1,-1)$   \\
  \bottomrule
\end{tabular}
\caption{$(a,b)$ for $\varphi_1(a,b)=\varphi_2(a,b)=0$}
\end{table} 
\end{proof}

Let $\psi_4,\psi_6,\chi_{10}$ and $\chi_{12} $ be the Igusa invariants defined in \cite{I35}.
The function $\psi_4$ ($\psi_6,\chi_{10},\chi_{12}$, resp.) gives a Siegel modular form of the character ${\rm id}$ and weight  $4$, ($6,10,12$, resp.) for the Siegel modular group $Sp_4(\mathbb{Z})$. 
They are algebraically independent.
Precisely, they have the following  expressions by the theta functions of (\ref{SiegelTheta}):
\begin{align*}
  \psi_4 =&2^{-2} \cdot \sum_{m=0}^9 \vartheta_m^8,\\
  \psi_6 = &2^{-2} \cdot \sum_{\text{syzygous}} \pm (\vartheta_{m_1} \vartheta_{m_2} \vartheta_{m_3})^4,\\
 \chi_{10} =& - 2^{-14} \cdot \prod_{m=0}^9 \theta_m^2 ,\\
  \chi_{12} =& 2^{-17} \cdot 3^{-1} \cdot \sum_{\text{complements of  G\"{o}pel quadruples}} (\vartheta_{m_1} \vartheta_{m_2}\vartheta_{m_3}\vartheta_{m_4} \vartheta_{m_5} \vartheta_{m_6})^4\\
  =& 2^{-17} \cdot 3^{-1} \cdot( (\vartheta_{4} \vartheta_{5} \vartheta_{6}\vartheta_{7} \vartheta_{8} \vartheta_{9})^4
  +(\vartheta_{1} \vartheta_{2} \vartheta_{3}\vartheta_{7} \vartheta_{8} \vartheta_{9})^4
  +(\vartheta_{0} \vartheta_{2} \vartheta_{5}\vartheta_{6} \vartheta_{7} \vartheta_{9})^4
  +(\vartheta_{0} \vartheta_{1} \vartheta_{4}\vartheta_{6} \vartheta_{8} \vartheta_{9})^4\\
  &+(\vartheta_{2} \vartheta_{3} \vartheta_{4}\vartheta_{6} \vartheta_{8} \vartheta_{9})^4
  +(\vartheta_{1} \vartheta_{3} \vartheta_{5}\vartheta_{6} \vartheta_{7} \vartheta_{9})^4
  +(\vartheta_{1} \vartheta_{2} \vartheta_{3}\vartheta_{4} \vartheta_{5} \vartheta_{6})^4
  +(\vartheta_{0} \vartheta_{2} \vartheta_{3}\vartheta_{5} \vartheta_{8} \vartheta_{9})^4\\
  &+(\vartheta_{0} \vartheta_{1} \vartheta_{3}\vartheta_{4} \vartheta_{7} \vartheta_{9})^4
  +(\vartheta_{0} \vartheta_{3} \vartheta_{4}\vartheta_{5} \vartheta_{7} \vartheta_{8})^4
  +(\vartheta_{0} \vartheta_{1} \vartheta_{2}\vartheta_{6} \vartheta_{7} \vartheta_{8})^4
  +(\vartheta_{0} \vartheta_{2} \vartheta_{3}\vartheta_{4} \vartheta_{6} \vartheta_{7})^4\\
  &+(\vartheta_{0} \vartheta_{1} \vartheta_{3}\vartheta_{5} \vartheta_{6} \vartheta_{8})^4
  +(\vartheta_{0} \vartheta_{1} \vartheta_{2}\vartheta_{4} \vartheta_{5} \vartheta_{9})^4
  +(\vartheta_{1} \vartheta_{2} \vartheta_{4}\vartheta_{5} \vartheta_{7} \vartheta_{8})^4).
\end{align*}
We have the local Fourier expansions of the Igusa invariants as follows.

\begin{prop}\label{PropIgusaFourier}
\begin{align*}
& \psi_4(W_0) = 1 +O(q_1,q_2),\\  
& \psi_6(W_0) = 1 +O(q_1,q_2), \\
& \chi_{10} (W_0) = - 2^{-2}\cdot q_1 q_2 ((\zeta +\zeta^{-1} -2) + O(q_1,q_2)),\\
& \chi_{12} (W_0) =  2^{-2} \cdot 3^{-1}\cdot q_1 q_2 ((\zeta +\zeta^{-1} +10) + O(q_1,q_2)).
\end{align*}
\end{prop}

\begin{proof}
From the explicit definitions of the Igusa  invariants and Proposition \ref{PropIgusaTheta}, the assertion follows.
\end{proof}

By the inverse period mapping for the Clingher-Doran family of (\ref{CDFamily}),
the pair of parameters  $(\alpha:\beta:\gamma:\delta)$ are expressed by the Igusa invariants.
The following  is the main theorem of \cite{CD}.

\begin{prop} (\cite{CD}, Theorem 3.5)\label{PropCDIgusa}
 $$
 (\alpha:\beta:\gamma:\delta) =(\psi_4 (W_0 ) :\psi_6 (W_0) :2^{12}\cdot 3^5 \cdot \chi_{10} (W_0) : 2^{12}\cdot 3^6\cdot \chi_{12} (W_0)).
 $$
\end{prop}

According to Proposition \ref{PropIgusaFourier} and \ref{PropCDIgusa}, we have the following Fourier expansions of the parameters of the Clingher-Doran family.

\begin{prop}\label{PropCDFourier}
\begin{align*}
& \alpha(W_0) = 1 +O(q_1,q_2),\\  
& \beta(W_0) = 1 +O(q_1,q_2), \\
& \gamma (W_0) = - 2^{10} \cdot 3^5 \cdot q_1 q_2 ((\zeta +\zeta^{-1} -2) + O(q_1,q_2)),\\
& \delta (W_0) =  2^{10}\cdot 3^{5} \cdot q_1 q_2 ((\zeta +\zeta^{-1} +10) + O(q_1,q_2)).
\end{align*}
\end{prop}

\subsection{Dern-Krieg theta functions}

Dern-Krieg \cite{DK} considered
the five theta functions $\Theta_k (W)$ $(k\in\{0,1,\cdots,4\})$ 
on the symmetric space $\mathbb{H}_I$ of (\ref{HI})
(see also \cite{FSMTheta}).
In this section, we will see their properties.

Let
$K$ be the imaginary quadratic field $\mathbb{Q}(\sqrt{-3})$.
For $W\in \mathbb{H}_I$, set
\begin{align}\label{DKTheta}
\Theta_k(W) =\sum_{g\in \mathfrak{O}_K^2 +\frac{1}{\sqrt{3}} p_k } e^{2\pi \sqrt{-1}  \overline{g} W  {}^t g}.
\end{align}
Here, the correspondence between the index $k$ and the characters $p_k$ is given in Table 6.
\begin{table}[h]
\center
\begin{tabular}{ccccccccccc}
\toprule
$k$ & $0$ &   $1$    &   $2$  & $3$ &  $4$  \\
 \midrule
$p_k=\begin{pmatrix} s\\t \end{pmatrix}$ & $\begin{pmatrix} 0 \\ 0  \end{pmatrix} $ & $\begin{pmatrix} 1 \\ 0  \end{pmatrix}$ & $\begin{pmatrix} 0 \\ 1  \end{pmatrix} $ &$\begin{pmatrix} 1 \\ -1  \end{pmatrix} $ & $\begin{pmatrix} 1 \\ 1  \end{pmatrix} $  \\
 \bottomrule
\end{tabular}
\caption{Correspondence between $k$ and $p_k$}
\end{table}

We  need the explicit expansion of (\ref{DKTheta})  on the loci $\{z=w\}$ and $\{z=-w\}$
in
 $ \mathbb{H}_I =\Big\{ W= \begin{pmatrix} \tau & z \\ w & \tau' \end{pmatrix} \Big\}$. 
By the direct calculation, we obtain the following.

\begin{prop}
\begin{align}\label{DKThetaz=w}
\Theta_k (W) |_{z=w}=&\sum_{(a,b,c,d)\in \mathbb{Z}^4} e^{2\pi \sqrt{-1} (a^2 - a c +c^2 +s c +\frac{s^2}{3})\tau + 2\pi \sqrt{-1}(b^2 -bd +d^2 +td +\frac{t^2}{3})\tau'} \notag\\
&\hspace{3cm} \times e^{2\pi \sqrt{-1} (2ab+2cd -ad -bc +tc +sd+\frac{2}{3} st))z }.
\end{align}
Also,
\begin{align}
\Theta_k (W) |_{z=-w}=&\sum_{(a,b,c,d)\in \mathbb{Z}^4} e^{2\pi \sqrt{-1} (a^2 - a c +c^2 +s c +\frac{s^2}{3})\tau + 2\pi \sqrt{-1} (b^2 -bd +d^2 +td +\frac{t^2}{3})\tau'} \notag \\
&\hspace{3cm}\times e^{2\pi \sqrt{-1} (\sqrt{3} \sqrt{-1} (ad-bc) +\frac{2\sqrt{-1}}{\sqrt{3}}(at-bs) -\frac{\sqrt{-1}}{\sqrt{3}} (ct-ds) )z }.
\end{align}
\end{prop}

\begin{rem}
The restriction (\ref{DKThetaz=w}) of (\ref{DKTheta})   on the locus $\{z=w\}$ is equal to the theta functions for the root lattice $A_2$, which are precisely studied in \cite{FSM}.
\end{rem}

Let us consider the Fourier expression of the Dern-Krieg theta functions.
Recall the notation (\ref{q})
and let
$
\xi=e^{\frac{\pi}{\sqrt{3}} z}.
$
We have the following.

\begin{prop}\label{PropDKThetaFourier}
(1)
\begin{align*}
& \Theta_0(W)|_{z=w} = 1+O(q_1,q_2), \\
& \Theta_1(W)|_{z=w} = q_1^{\frac{1}{3}}(3+O(q_1,q_2)),\\
& \Theta_2(W)|_{z=w} = q_2^{\frac{1}{3}}(3+O(q_1,q_2)),\\
& \Theta_3(W)|_{z=w} = q_1^{\frac{1}{3}} q_2^{\frac{1}{3}}(3(\zeta^{-\frac{2}{3}} +2\zeta^{\frac{1}{3}})+O(q_1,q_2)),\\
& \Theta_4(W)|_{z=w} = q_1^{\frac{1}{3}} q_2^{\frac{1}{3}}(3(\zeta^{\frac{2}{3}} +2\zeta^{-\frac{1}{3}})+O(q_1,q_2)).
\end{align*}

(2)
\begin{align*}
& \Theta_0(W)|_{z=-w} = 1+O(q_1,q_2), \\
& \Theta_1(W)|_{z=-w} = q_1^{\frac{1}{3}}(3+O(q_1,q_2)),\\
& \Theta_2(W)|_{z=-w} = q_2^{\frac{1}{3}}(3+O(q_1,q_2)),\\
& \Theta_3(W)|_{z=-w} = q_1^{\frac{1}{3}} q_2^{\frac{1}{3}}(3(1+\xi^2 +\xi^{-2})+O(q_1,q_2)),\\
& \Theta_4(W)|_{z=-w} = q_1^{\frac{1}{3}} q_2^{\frac{1}{3}}(3(1+\xi^2 +\xi^{-2})+O(q_1,q_2)).
\end{align*}

\end{prop}

\begin{proof}
$\Theta_k (W)$ has an expression 
$$
\Theta_k(Z) =q_1^{\alpha_1} q_2^{\alpha_2} \sum_{(a,b,c,d)\in \mathbb{Z}^4} e^{2\pi\sqrt{-1} (\varphi_1(a,b,c,d) \tau +\varphi_2(a,b,c,d)\tau')} \times e^{2\pi \sqrt{-1}( \psi_1(a,b,c,d) z+ \psi_2 (a,b,c,d)w)},
$$
where $\varphi_l(a,b,c,d),\psi_l(a,b,c,d) \in \mathbb{C}[a,b,c,d]$ $(l\in\{1,2\})$ and  $\varphi_l(a,b,c,d)$ does not contain any constant terms.
For each $k$, $(a,b,c,d)$ satisfying $\varphi_1(a,b,c,d)=\varphi_2(a,b,c,d)=0 $ is given in Table 7.
From that table, we can obtain the assertion.

\begin{table}[h]
\center
\begin{tabular}{ccccccccccc}
\toprule
$k$ & $\alpha_1$ & $\alpha_2$ & $(a,b,c,d)$ for $\varphi_1(a,b,c,d)=\varphi_2(a,b,c,d)=0 $  \\
 \midrule
 $0$ & $0$ & $0$ & $(0,0,0,0)$  \\
$1$ & $\frac{1}{3}$  & $0$ & $(0,0,0,0),(0,0,-1,0),(-1,0,-1,0)$  \\
 $2$ & $0$ & $\frac{1}{3}$ & $(0,0,0,0),(0,0,0,-1),(0,-1,0,-1)$ \\
 $3$ & $\frac{1}{3}$  & $\frac{1}{3}$ & $(0,0,0,0),(0,0,0,1),(0,1,0,1),(0,0,-1,0),(0,0,-1,1),$   \\
 & &                                                   & $(0,1,-1,1),(-1,0,-1,0),(-1,0,-1,1),(-1,1,-1,1)$\\
 $4$ & $\frac{1}{3}$  & $\frac{1}{3}$ & $(0,0,0,0),(0,0,0,-1),(0,-1,0,-1),(0,0,-1,0),(0,0,-1,-1),$ \\
 & &                                                   & $(0,-1,-1,-1),(-1,0,-1,0),(-1,0,-1,-1),(-1,-1,-1,-1)$\\
 \bottomrule
\end{tabular}
\caption{ $(a,b,c,d)$ for $\varphi_1(a,b,c,d) =\varphi_2(a,b,c,d) = 0$}
\end{table} 

\end{proof}

\subsection{Burkhardt invariants}

Let us consider the actions of the system in (\ref{Action}) on the theta functions.
Setting $K=\mathbb{Q}(\sqrt{-3})$,
let us consider a mapping $\Psi: \Gamma(\mathfrak{O}_K) \rightarrow GL_5(\mathbb{C})$ 
 satisfying
\begin{align}
({\rm det}M)^{-1}
\begin{pmatrix}
\Theta_0|_1 {M} \\ \Theta_1|_1 {M} \\ \Theta_2|_1 {M} \\ \Theta_3|_1 {M} \\ \Theta_4|_1 {M}
\end{pmatrix}
=\Psi(M) \begin{pmatrix}
\Theta_0 \\ \Theta_1 \\ \Theta_2 \\ \Theta_3 \\ \Theta_4
\end{pmatrix},
\end{align}
for $M\in \Gamma(\mathfrak{O}_K)$.
Here, we use the notation (\ref{ModularFormTrans}).
We have the following explicit expression of $\Psi$.

\begin{prop}(\cite{DK} Theorem 8)
Set $\omega=e^{\frac{2 \pi \sqrt{-1}}{3}}$.
Then, for the matrices in (\ref{Action}),
\begin{align*}
\Psi(M_1)={\rm diag}(1,\omega,1,\omega,\omega), 
\quad \Psi(M_2)={\rm diag}(1,1,\omega,\omega,\omega), 
\quad \Psi(M_3)={\rm diag}(1,1,1,\omega,\omega^2)
\end{align*}
and
\begin{align*}
\Psi(J)=\frac{1}{3}
\begin{pmatrix}
-1 & -2 & -2 & -2 & -2 \\
-1 & 1 & -2 & 1 & 1\\
-1 & -2 & 1& 1 & 1\\
-1 & 1 & 1 & 1 & -2 \\
-1 & 1 & 1 & -2 & 1
\end{pmatrix}.
\end{align*}
\end{prop}

We remark that
$\Gamma(\mathfrak{O}_K)/ H$, 
where $H $ $(\subset \Gamma(\mathfrak{O}_K) )$ is the congruence subgroup of level $\sqrt{-3}$,
 is a simple group of order $25920$.
 This group is called the Burkhardt group.
The above mapping $\Psi $ defines an irreducible $5$-dimensional representation of the Burkhardt group.
The Molien series of this representation is calculated as
$$
\frac{1+t^{45}}{(1-t^4) (1-t^6) (1-t^{10}) (1-t^{12}) (1-t^{18})}.
$$
Letting $G_{33}$ be the group of No.33 in the list of Shephard-Todd \cite{ST} of complex reflection group,
the Burkhardt group is isomorphic to $G_{33}/\{\pm 1\}$.

Let us consider the polynomial ring $\mathbb{C}[T_0,T_1,T_2,T_3,T_4]$ in indeterminates $T_0,\cdots,T_4$.
For integers $n_0,\cdots, n_4$ such that $n_0\geq 0$ and $n_1\geq n_2 \geq n_3 \geq n_4 \geq 0$,
we set
$$
(n_0; n_1,n_2,n_3,n_4) =T_0^{n_0} \sum_{\sigma \in \mathfrak{S}_4 /{\rm Stab} (n_1,n_2,n_3,n_4)} \prod_{i=1,2,3,4} T_{\sigma(i)}^{n_i} \in \mathbb{C}[T_0,T_1,T_2, T_3 ,T_4],
$$
where $\mathfrak{S}_4$ is the symmetric group and 
$${\rm Stab} (n_1,n_2,n_3,n_4)= \{\sigma\in \mathfrak{S}_4 | (n_{\sigma(1)},n_{\sigma(2)},n_{\sigma(3)},n_{\sigma(4)})=(n_1,n_2,n_3,n_4) \}.$$
However, unless $n_1=0$,  we omit entries $0$ for $n_j$ $(j=1,2,3,4)$ in this notation.  
For example, 
\begin{align*}
& (4;0)=T_0^4,\quad \quad  (1;3)=T_0 (T_1^3+T_2^3 +T_3^3 +T_4^3),\\
& (0;3,3)=T_1^3 T_2^3+ T_1^3 T_3^3 +T_1^3 T_4^3 + T_2^3 T_3^3 +T_2^3 T_4^3 + T_3^3 T_4^3.
\end{align*}
Burkhardt \cite{B} determines the algebraically independent  invariants $B_4,B_6,B_{10},B_{12}$ and $B_{18}$ of the representation $\Psi$ as follows:
\begin{align*}
B_4(T_0,\cdots,T_4) =& (4;0) +8(1;3) +48 (0;1,1,1,1), \\
B_6(T_0,\cdots,T_4)=& (6;0) -20 (3;3) +360(2;1,1,1,1)+80 (0;3,3) -8 (0;6), \\
B_{10}(T_0,\cdots,T_4) =& (6;1,1,1,1) -(4;3,3) +(3;4,1,1,1) +9(2;2,2,2,2) + (1;6,3) -6(1;3,3,3)\\
& -2(0;7,1,1,1)+2(0;4,4,1,1), \\
B_{12}(T_0,\cdots,T_4) =& 3(8;1,1,1,1) +5(6;3,3) -33(5;4,1,1,1) +243(4;2,2,2,2) -(3;6,3) -102(3;3,3,3) \\
& +30(2;7,1,1,1) +78(2;4,4,1,1) -108(1;5,2,2,2) -4(0;9,3) +16(0;6,6)\\
& -8(0;6,3,3) +168(0;3,3,3,3),
\end{align*}
and
\begin{align*}
B_{18}(T_0,\cdots,T_4) = & 3(10;2,2,2,2) -4(9;3,3,3) +6 (8,4,4,1,1) -18(7;5,2,2,2) -(6;6,6) +10(6;6,3,3)\\
             &+96(6;3,3,3,3)-12(5;7,4,1,1) -90(5;4,4,4,1) +27(4;8,2,2,2) +108 (4;5,5,2,2)  \\
             &+2(3;9,6) -8(3;9,3,3)+4(3;6,6,3) -168(3;6,3,3,3) +6(2;10,4,1,1) \\
             &-24(2;7,7,1,1)+12(2;7,4,4,1) +315(2;4,4,4,4)-12(1;11,2,2,2)\\
             &+18(1;8,5,2,2)-72(1;5,5,5,2) -(0;12,6)+2(0;12,3,3)+2(0;9,9)-2(0;9,6,3)\\
             &-8(0;9,3,3,3)+6(0;6,6,6)+8(0;6,6,3,3).
\end{align*}

\begin{rem}
In the original definition of these invariants in  \cite{B},
there are some typos.
For example, 
monomials ``$+12 Y_0^8 \sum Y_1^4 Y_2^4 Y_3 Y_4$'', ``$-36 Y_0^7 \sum Y_1^5 Y_2^2 Y_3^3 Y_4^3$'', ``$+12 Y_0 \sum Y_1^{11} Y_2^2  Y_3^2 Y_4^2$'' and \\
``$-27 Y_0 \sum Y_1^5 Y_2^5 Y_3^5 Y_4^2$'' in the original $J_{18}$ (\cite{B} p.209)
are incorrect.
\end{rem}

From now on, 
we will use the notation 
$$
B_j (W) = B_j (\Theta_0(W),\Theta_1(W),\Theta_2(W),\Theta_3(W),\Theta_4(W)) \quad (j\in \{4,6,10,12,18\}),
$$
where $\Theta_k (W)$ are the theta functions of (\ref{DKTheta}).
Especially, $W\mapsto B_j(W)$ is a holomorphic function on  $\mathbb{H}_I$.

\begin{rem}
The above explicit expressions of $B_4(W)$ and $B_6(W)$ are equal to $E_4$ and $E_6$ in  \cite{DK} Section 5.
We remark that Dern and Krieg obtain them by using a calculation aided by MAGMA.
On the other hand, those of $B_{10} (W), B_{12} (W)$ and $B_{18} (W)$ do not appear in that paper.
\end{rem}

We will need the Fourier expansions of these functions.
Due to Proposition \ref{PropDKThetaFourier} and the definition of the Burkhardt invariants, we have the following.

\begin{prop}\label{PropBFourier}
(1)
\begin{align*}
& B_4(W)|_{z=w} = 1+O(q_1,q_2),\\
& B_6(W)|_{z=w} = 1+O(q_1,q_2),\\
& B_{10}(W)|_{z=w} = 2 \cdot 3^4 \cdot q_1 q_2 ((\zeta+\zeta^{-1}-2) + O(q_1,q_2)),\\
& B_{12}(W)|_{z=w} = 2 \cdot 3^5 \cdot q_1 q_2 ((\zeta+\zeta^{-1} +10 ) + O(q_1,q_2)),\\
&B_{18}(W)|_{z=w} =0. 
\end{align*}

(2)
\begin{align*}
& B_4(W)|_{z=-w} = 1+O(q_1,q_2),\\
& B_6(W)|_{z=-w} = 1+O(q_1,q_2),\\
& B_{10}(W)|_{z=-w} = 3^4 \cdot q_1 q_2 ((\xi^4 +\xi^{-4} +2\xi^2 +2\xi^{-2} -6) + O(q_1,q_2)),\\
& B_{12}(W)|_{z=-w} = 3^5 \cdot q_1 q_2 ((\xi^4 +\xi^{-4} +2\xi^2 +2\xi^{-2} +18) + O(q_1,q_2)),\\
&B_{18}(W)|_{z=-w} = 3^9 \cdot q_1^2 q_2^2 ((\xi-\xi^{-1})^6 (\xi+\xi^{-1})^2 +O(q_1,q_2)). 
\end{align*}
\end{prop}

\begin{rem}
The relation $B_{18}(W)|_{z=w} =0 $ is equal to the relation $P_{18} (A_1,\cdots,A_5)=0$ in \cite{FSM} p.31-32.
\end{rem}

\subsection{Theta expression of inverse period mapping}

In this section, let us give the explicit expression of the parameters $t_4,t_6,t_{10},t_{12}, t_{18}$ of (\ref{S(t)}) by the Dern-Krieg theta functions. 
We will consider the composite of 
the modular isomorphism $f^{-1}$ of (\ref{ModularIso}) 
and the modular form $t_j$ in Theorem \ref{ThmModular}.
By abuse of notation,
$\mathbb{H}_I \ni W\mapsto t_j(W) \in \mathbb{C}$ stands for this function.

First, 
let us obtain the explicit expression of them on the locus $\{z=w\}$ in $\mathbb{H}_I$.

\begin{lem}\label{Proptzw}
\begin{align*}
&(t_4(W)|_{z=w}:t_6(W)|_{z=w}:t_{10}(W)|_{z=w}:t_{12}(W)|_{z=w})\\
&=(-3 \cdot B_4(W)|_{z=w} : -2 \cdot  B_6(W)|_{z=w}: 2^9 \cdot 3 \cdot B_{10}(W)|_{z=w} : 2^9 \cdot B_{12}(W)|_{z=w}).
\end{align*}
\end{lem}

\begin{proof}
Due to the property of Siegel modular forms for $Sp_4(\mathbb{Z})$ and Proposition \ref{PropCDIgusa},
the restrictions of the  Dern-Krieg theta functions on the locus $\{z=w\}$
can be
expressed by
 $\alpha, \beta, \gamma$ and $\delta$ of the family (\ref{CDFamily}). 
 For example,
 $B_{10}(W)|_{z=w}$ should be given by a linear combination of $\gamma$ and $\alpha\beta$.
Moreover, by using Proposition \ref{PropCDFourier} and Proposition \ref{PropBFourier} (1), 
we can obtain the following exact expression:
\begin{align*}
(\alpha:\beta:\gamma:\delta)=(B_4(W)|_{z=w} : B_6(W)|_{z=w}: -2^9\cdot 3 \cdot B_{10}(W)|_{z=w} : 2^9 \cdot B_{12}(W)|_{z=w}).
\end{align*}
By recalling  Proposition \ref{PropS(t)CD},
we have  the assertion.
\end{proof}

\begin{thm}\label{ThmThetaExpression}
\begin{align*}
&(t_4(W):t_6(W):t_{10}(W):t_{12}(W):t_{18}(W))\\
&=(-3 \cdot B_4(W): -2 \cdot B_6 (W): 2^9 \cdot 3 \cdot B_{10} (W) : 2^9 \cdot B_{12} (W): -2^{16} \cdot B_{18} (W) ).
\end{align*}
\end{thm}

\begin{proof}
By considering the degeneration of our $K3$ surface $S([t])$ on the locus $\{z=w\}$ as in Proposition \ref{PropReflection} (2),
together with  Proposition \ref{PropBFourier} (1) and Lemma \ref{Proptzw},
we can take a constant $C_{18}$ such that  
\begin{align*}
&(t_4(W):t_6(W):t_{10}(W):t_{12}(W):t_{18}(W))\\
&=(-3 \cdot B_4(W): -2\cdot B_6 (W): 2^9 \cdot 3 \cdot B_{10} (W) : 2^9 \cdot B_{12} (W): C_{18} \cdot B_{18} (W) ).
\end{align*}

On the other hand, by  Proposition \ref{PropModularFormIso},
the modular form $d_{90}(t)$ of (\ref{d90})  should be zero on the locus ${z=-w}$.
According to (\ref{d90}) and Proposition \ref{PropBFourier} (2),
the leading term of the Fourier expansion of
$
d_{90}(W)|_{z=-w}
$ 
is calculated as
$$
q_1^7 q_2^7  \times (C_{18} + 2^{16}) \times (\text{ a polynomial in } \xi , \xi^{-1}).
$$
Since this leading term should be equal to $0$,
it follows $C_{18}=-2^{16}$.
\end{proof}

By the above theorem,
together with the definitions of the Dern-Krieg theta functions
and the Burkhardt invariants,
we have the following result. 

\begin{cor}\label{CorQ-rational}
The modular forms
$W \mapsto t_j(W)$ $(j\in\{4,6,10,12,18\})$ on $\mathbb{H}_I$,
which are given as the inverse of the period mapping for the family of $K3$ surface $S([t])$,
 are $\mathbb{Q}$-rational.
\end{cor}

$\mathbb{Q}$-rational modular forms are very important in number theory.
In fact,
we can obtain  arithmetic results of 
the family of elliptic curves in (\ref{EllipticWeierstrass})
and
the (normalized) Eisenstein series $g_j(\tau)$ $(j=2,3)$
from the relation 
between algebraic varieties and $\mathbb{Q}$-rational modular forms.
Therefore, this corollary suggests that
our family of $K3$ surfaces will be a good model
for arithmetic researches  of $K3$ surfaces.

\section*{Acknowledgment}
The first author would like to express his gratitude to Prof. Manabu Oura for valuable discussions about complex reflection groups.
The first author  is supported by JSPS Grant-in-Aid for Scientific Research (18K13383)
and 
MEXT LEADER.
The second author is supported by JSPS Grant-in-Aid for Scientific Research (19K03396).

\vspace{3mm}

\begin{center}
\hspace{7.7cm}\textit{Atsuhira  Nagano}\\
\hspace{7.7cm}\textit{Faculty of Mathematics and Physics}\\
 \hspace{7.7cm} \textit{Institute of Science and Engineering}\\
\hspace{7.7cm}\textit{Kanazawa University}\\
\hspace{7.7cm}\textit{Kakuma, Kanazawa, Ishikawa}\\
\hspace{7.7cm}\textit{920-1192, Japan}\\
 \hspace{7.7cm}\textit{(E-mail: atsuhira.nagano@gmail.com)}
  \end{center}

\vspace{1mm}

\begin{center}
\hspace{7.7cm}\textit{Hironori  Shiga}\\
\hspace{7.7cm}\textit{Graduate School of Science }\\
\hspace{7.7cm}\textit{Chiba University}\\
\hspace{7.7cm}\textit{Yayoi-cho 1-33, Inage-ku, Chiba}\\
\hspace{7.7cm}\textit{263-8522, Japan}\\
 \hspace{7.7cm}\textit{(E-mail: shiga@math.s.chiba-u.ac.jp)}
  \end{center}

\end{document}